\documentclass[10pt]{amsart}

\usepackage{amssymb,amsmath,amsthm,amsfonts,microtype}
\usepackage{graphicx}
\usepackage{pdfsync}

\newcommand{\comment}[1]{}

\newtheorem{lem}{Lemma}
\newtheorem{propn}{Proposition}
\newtheorem{cor}{Corollary}
\newtheorem{thm}{Theorem}

\theoremstyle{remark}

\theoremstyle{definition}
\newtheorem{defn}{Definition}

\newcommand{\R}{\mathbb R}

\newcommand{\Z}{\mathbb Z}
\newcommand{\T}{\mathbb T}

\newcommand{\N}{\mathbb N}
\newcommand{\C}{\mathbb C}

\newcommand{\MM}{\mathcal M}
\newcommand{\NN}{\mathcal N}

\DeclareMathOperator{\lcm}{lcm}

\newcommand{\vp}{\varphi}
\newcommand{\D}{\delta}
\newcommand{\de}{\delta}

\newcommand{\si}{\sigma}

\newcommand{\VE}{\varepsilon}
\newcommand{\A}{\alpha}
\newcommand{\B}{\beta}
\newcommand{\lm}{\lambda}
\newcommand{\Lm}{\Lambda}

\newcommand{\be}{\begin{equation}}
\newcommand{\ee}{\end{equation}}
\newcommand{\bee}{\begin{equation*}}
\newcommand{\eee}{\end{equation*}}

\begin{document}
\title{Distances and Trees in Dense subsets of $\mathbb{Z}^d$}
\author{Neil Lyall\quad\quad\quad\'Akos Magyar}
\thanks{The first and second authors were partially supported by grants NSF-DMS 1702411 and NSF-DMS 1600840, respectively.}

\address{Department of Mathematics, The University of Georgia, Athens, GA 30602, USA}
\email{lyall@math.uga.edu}
\email{magyar@math.uga.edu}


\begin{abstract}
In \cite{FKW} Katznelson and Weiss establish that all sufficiently large distances can always be attained between pairs of points from any given measurable subset of $\mathbb{R}^2$ of positive upper (Banach) density. A second proof of this result, as well as a stronger ``pinned variant'', was given by Bourgain in \cite{B} using Fourier analytic methods.
In \cite{M1} the second author adapted Bourgain's  Fourier analytic approach to established a result analogous to that of Katznelson and Weiss for subsets $\mathbb{Z}^d$ provided $d\geq 5$. 
We present a new direct proof of this discrete distance set result and generalize this to arbitrary trees. Using appropriate discrete spherical maximal function theorems we ultimately establish the natural ``pinned variants'' of these results.
\end{abstract}
\maketitle

\setlength{\parskip}{5pt}


\section{Introduction}



\subsection{Distance sets and existing results}

A result of Katznelson and Weiss \cite{FKW} states that all sufficiently large distances can  always be attained between pairs of points from any given measurable subset of $\mathbb{R}^2$ of positive upper (Banach) density. 
Specifically, if $A$ is a measurable subset of $\R^2$ of positive upper Banach density, they established the existence of a threshold $\lm_0=\lm_0(A)$ such that the distance set
\[\text{dist}(A)=\{|x-y|\,:\, x,y\in A\}\supseteq[\lm_0,\infty).\]

Recall that the \emph{upper Banach density} $\D^*(A)$ of a set $A\subseteq\R^d$ is defined by \[\D^*(A):=\lim_{N\rightarrow\infty}\sup_{t\in\R^d}\frac{|A\cap(t+Q_N)|}{|Q_N|},\]
where $|\cdot|$ denotes Lebesgue measure on $\R^d$ and $Q_N$ denotes the cube $[-N/2,N/2]^d$.

This result was later established using Fourier analytic methods by Bourgain in \cite{B}.  Bourgain also established a ``pinned variant'', namely that for any $\lm_1\geq\lm_0$ there is a fixed $x\in A$ such that
\[\text{dist}(A;x)=\{|x-y|\,:\, y\in A\}\supseteq[\lm_0,\lm_1].\]

In \cite{M1} the second author adapted Bourgain's  Fourier analytic approach to established a result analogous to that of Katznelson and Weiss for subsets $\mathbb{Z}^d$, namely that if $A\subseteq \Z^d$  of positive upper Banach density and $d\geq5$, then there exists $\lm_0=\lm_0(A)$ and an integer $q$, depending  on $d$ and the density of $A$, such that 
\[\text{dist}^2(A)=\{|x-y|^2\,:\, x,y\in A\}\supseteq[\lm_0,\infty)\cap q^2\Z.\] 

Recall that the \emph{upper Banach density} $\D^*(A)$ of a set $A\subseteq\Z^d$ is analogously defined by \[\D^*(A):=\lim_{N\rightarrow\infty}\sup_{t\in\Z^d}\frac{|A\cap(t+Q_N)|}{|Q_N|},\]
where, $|\cdot|$ now denotes counting measure on $\Z^d$ and $Q_N$  the discrete cube $[-N/2,N/2]^d\cap\Z^d$.

Note that since $A$ could fall entirely into a fixed congruence class of some integer $1\leq r\leq\delta^*(A)^{-1/d}$  the value of $q$ in the result above must be divisible by the  least common multiple of all integers $1\leq r\leq \delta^*(A)^{-1/d}$.


\subsection{New results}
We will denote, for any integer $\lm$, the discrete sphere of radius $\sqrt{\lm}$ by $S_\lm$, namely
\[S_{\lambda}:=\{x\in\mathbb{R}^d\,:\, |x|^2=\lambda\}\cap\mathbb{Z}^d.\]

In this paper we will present a  new direct  proof of the following discrete distance set result from \cite{M1}.

\begin{thm}[Unpinned Distances]\label{unpinned}
Let $A\subseteq\mathbb{Z}^d$ with $d\geq5$ and $\D^*(A)>0$.

There exist  $q=q(\D^*(A))$ and $\lm_0=\lm_0(A)$ such that for any integer  $\lambda\geq \Lm_0$ there exist a pair of points 
\bee\{x,x+x_1\}\subseteq A\quad\text{with}\quad |x_1|^2=q^2\lm.\eee
In fact, for any $\VE>0$ there exist $q=q(\VE,d)$ and $\Lm_0=\Lm_0(A,\VE)$ such that for any integer $\lambda\geq \Lm_0$ one has 
\bee
\frac{|A\cap(x+qS_{\lambda})|}{|S_\lm|}>\D^*(A)-\VE  \quad\text{for some $x\in A$}.\eee
\end{thm}


By considering sets $A$ of the form $\bigcup_{s\in\{1,\dots,q\}^d}A_s$ with each set $A_s$ a ``random" subset of the congruence class $s+(q\Z)^d$ one can easily see that the second conclusion above is best possible or  ``$\VE$-optimal".

\smallskip

The first main new result of this paper is  the following  ``pinned variant" of Theorem \ref{unpinned} above, in other words a discrete analogue of Bourgain's pinned distances theorem in \cite{B}.

\begin{thm}[Pinned Distances]\label{pinned}
Let $\VE>0$ and  $A\subseteq\mathbb{Z}^d$ with $d\geq5$.


There exist $q=q(\VE,d)$ and $\Lm_0=\Lm_0(A,\VE)$ such that for any $\Lm_1\geq \Lm_0$ there exists a fixed $x\in A$ such that
\bee
\frac{|A\cap(x+qS_{\lambda})|}{|S_\lm|}>\D^*(A)-\VE\quad\text{for all integers} \quad \Lm_0\leq\lm\leq\Lm_1.\eee
\end{thm}




\medskip

Our approach to Theorems \ref{unpinned} and \ref{pinned}  allows us to establish generalizations  to arbitrary finite trees.

\smallskip

\begin{defn}[Trees and Down-Labeled Trees]
A \emph{tree} $\Gamma=\Gamma(V,E)$ is a  connected acyclic graph.
It is easy to verify that the vertex set $V$ and edge set $E$ of any given finite tree $\Gamma=\Gamma(V,E)$ must satisfy $|E|=|V|-1$ and that there exists
an enumeration  of $V=\{v_0,v_1,\dots,v_n\}$ so that 
each edge $e_j$ from  $E=\{e_1,\dots,e_n\}$ takes the form $e_j=\{v_i,v_j\}$ for some unique $0\leq i< j$.  
We shall refer to a tree $\Gamma$ on $n+1$ vertices as \emph{down-labeled} if its vertex set $\{v_0,v_1,\dots,v_n\}$ has been enumerated as described above and use, for each $1\leq j\leq n$, the convenient notation $i_\Gamma(j)$ to denote the unique $0\leq i< j$ for which $\{v_i,v_j\}\in E$.
\end{defn}

\smallskip

\begin{figure}[htbp] 
\centering
\includegraphics[width=2.5in]{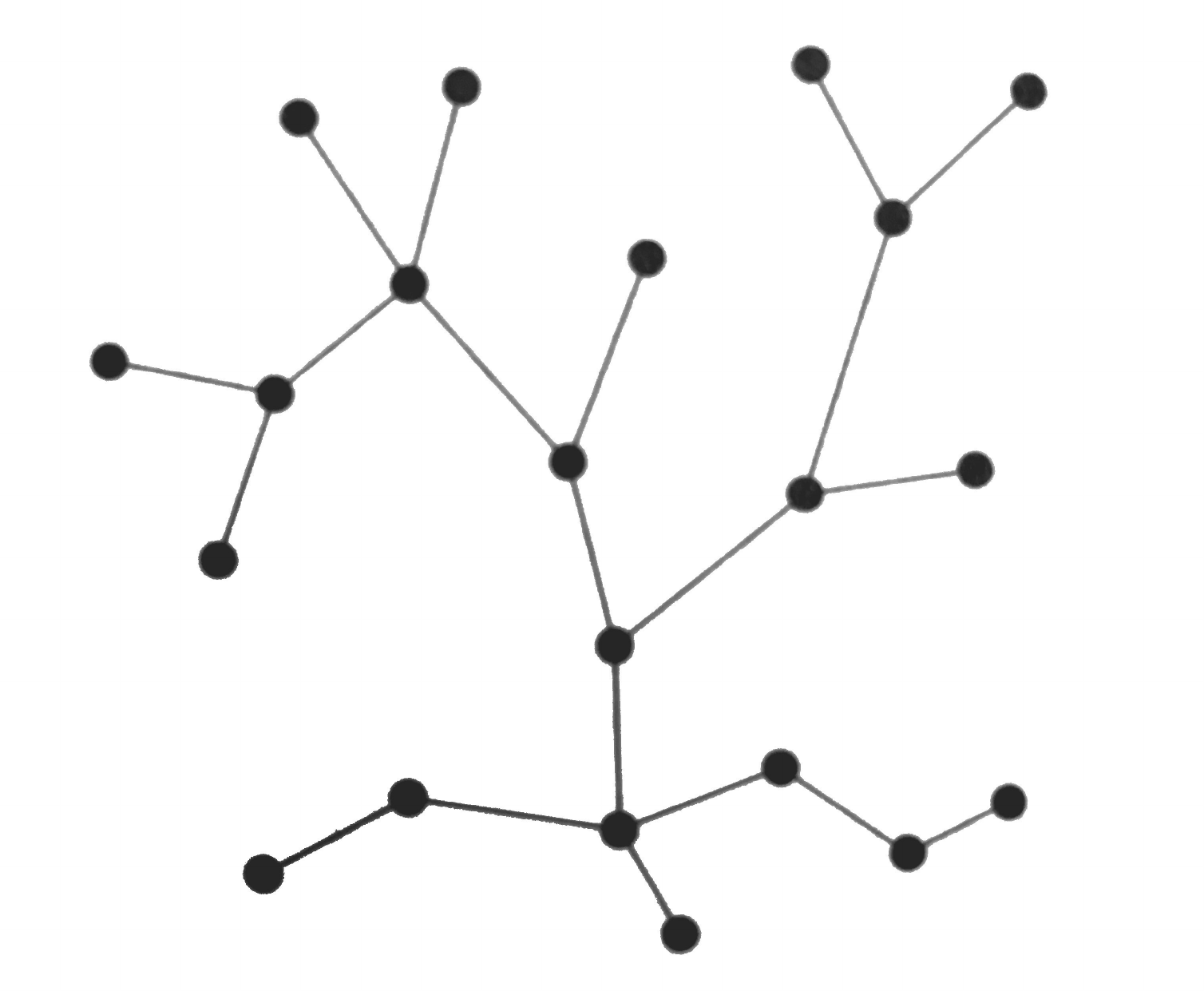}
\caption{A tree}
\label{}
\end{figure}





The following result  gives the aforementioned  generalizations of Theorem \ref{unpinned} to arbitrary finite trees.


\begin{thm}[Unpinned Trees]\label{unpinnedTree}

Let $\Gamma$  be a down-labeled tree on $n+1$ vertices.

If $A\subseteq\mathbb{Z}^d$ with $d\geq5$ and $\D^*(A)>0$, then there exist  $q=q(\D^*(A))$ and $\Lm_0=\Lm_0(A,n)$ such that for any integers $\lambda_1,\dots,\lm_n\geq \Lm_0$ there exists  
\bee
\{x+x_0,x+x_1,\dots,x+x_n\}\subseteq A \text{ \ with  $x_0=0$ and $|x_j-x_{i_\Gamma(j)}|^2=q^2\lm_j$ for all $1\leq j\leq n$.}\eee
In fact, for any $\VE>0$
there exist $q=q(\VE,d)$ and $\Lm_0=\Lm_0(A,n,\VE)$ such that for any integers $\lambda_1,\dots,\lm_n\geq \Lm_0$  
\bee
\frac{\left|\left\{\{x_1,\dots,x_n\}\subseteq A-x\,:\, \text{$x_j-x_{i_\Gamma(j)}\in qS_{\lm_j}$ for $1\leq j\leq n$}\right\}\right|}{|S_{\lm_1}|\cdots|S_{\lm_n}|}>\D^*(A)^n-\VE 
\eee
for some $x\in A$, with the understanding that $x_0=0$.
\end{thm}

We remark that no result of this type can possibly hold with a threshold $\Lm_0$ independent of $n$ and that a quantitatively weaker (and non-optimal) version of Theorem \ref{unpinnedTree}, in which the parameter $q$ also depended on $n$,  was previously established by Bulinski in \cite{BU} using methods from Ergodic theory.

\smallskip

Our main result, which we stress does not follow from the arguments presented in \cite{BU},  is the following ``pinned variant" of Theorem \ref{unpinnedTree} which generalizes Theorem \ref{pinned} above.

\begin{thm}[Pinned Trees]\label{pinnedTree}
Let $\Gamma$  be a down-labeled tree on $n+1$ vertices, $\VE>0$, and $A\subseteq\mathbb{Z}^d$ with $d\geq5$.

There exist $q=q(\VE,d)$ and $\Lm_0=\Lm_0(A,\VE)$
 such that for any $\Lm_1\geq \Lm_0$ there exists a fixed $x\in A$ such that
\bee
\frac{\left|\left\{\{x_1,\dots,x_n\}\subseteq A-x\,:\, \text{$x_j-x_{i_\Gamma(j)}\in qS_{\lm_j}$ for $1\leq j\leq n$}\right\}\right|}{|S_{\lm_1}|\cdots|S_{\lm_n}|}>\D^*(A)^n-\VE\eee
for any choice of integers $\lambda_1,\dots,\lm_n$ that satisfy  \ $\Lm_0\leq\lambda_1,\dots,\lm_n\leq\Lm_1$, with the understanding that $x_0=0$.
\end{thm}

\smallskip

\subsection{Outline of paper}\

In Section \ref{Sect2}  we state analogues of Theorems \ref{unpinned}-\ref{pinnedTree} for uniformly distributed subsets of $\Z^d$ and reduce their proofs to that of analogous results for uniformly distributed compact subsets of $\Z^d$. 

In Section \ref{P0Proof} we complete the proofs of Theorems \ref{unpinned} and \ref{unpinnedTree} by proving the analogous result for uniformly distributed compact subsets of $\Z^d$, namely Proposition \ref{Propn0}. To do this we
introduce a norm which measures the uniformity of distribution within residue classes modulo $q$ with respect to a scale $L$. We then prove that this norm controls the frequency with which certain distances appear  in compact subset of $\Z^d$, this is analogous to the so-called von-Neumann type inequalities in additive combinatorics. We ultimately demonstrate that this control on the frequency with which certain distances appear allows  us to actually control the frequency with which trees with certain prescribed edge lengths appear in compact subset of $\Z^d$. 

In Sections \ref{P1Proof} and \ref{MollyProof} we complete the proofs of Theorems \ref{pinned} and \ref{pinnedTree} by proving the analogous result for uniformly distributed compact subsets of $\Z^d$, namely Proposition \ref{Propn1}. In Section \ref{P1Proof} we reduce matters to the \emph{Discrete Spherical Maximal Function Theorem} of Magyar, Stein  and Wainger \cite{MSW} and a closely related ``mollified variant" thereof, namely Proposition \ref{mollified}, whose  statement and proof  we  presented in Section \ref{MollyProof}.

\smallskip

\section{Reduction to Uniformly Distributed Compact Subsets of $\Z^d$}\label{Sect2}

\subsection{Distances and Trees in Uniformly Distributed Subsets of $\Z^d$}


\begin{defn}[Definition of $q_\eta$ and $\eta$-uniform distribution] For any $\eta>0$ we define
\bee
q_\eta:=\lcm\{1\leq q\leq C\eta^{-2}\}\eee
with $C>0$ a (sufficiently) large absolute constant
and $A\subseteq\Z^d$  to be \emph{$\eta$-uniformly distributed (modulo $q_\eta$)}
 if its \emph{relative} upper Banach density on any residue class modulo $q_\eta$ never exceeds $(1+\eta^2)$ times its density on $\Z^d$, namely if for all $s\in\{1,\dots,q_\eta\}^d$ one has
\[\D^*(A\,|\,s+(q_\eta \Z)^d)\leq(1+\eta^2)\,\D^*(A).\]
\end{defn}

Theorems \ref{unpinned} and \ref{unpinnedTree} are immediate consequences, via an easy density increment argument, of the following analogous result for uniformly distributed sets.

\begin{thm}[Theorems \ref{unpinned} and \ref{unpinnedTree} for Uniformly Distributed Sets]\label{unpinnedUD}\

Let $\VE>0$ and $A\subseteq \Z^d$ with $d\geq 5$ be $\eta$-uniformly distributed for some $\eta>0$.

\begin{itemize}
\item[\text{(i)}]
If $0<\eta\ll\VE^2$, then there exist $\Lm_0=\Lm_0(A,\eta)$ such that for any integer $\lambda\geq \Lm_0$ one has 
\bee
\frac{|A\cap(x+S_{\lambda})|}{|S_\lm|}
>\D^*(A)-\VE \quad\text{for some $x\in A$}\eee


\item[\text{(ii)}]
If $0<\eta\ll\VE^2/n$, then for any  down-labeled tree $\Gamma$ on $n+1$ vertices there exist $\Lm_0=\Lm_0(A,n,\eta)$ such that for any integers  $\lambda_1,\dots,\lm_n\geq \Lm_0$  there exist $x\in A$ for which
\bee
\frac{\left|\left\{\{x_1,\dots,x_n\}\subseteq A-x\,:\, \text{$|x_j-x_{i_\Gamma(j)}|^2=\lm_j$ for $1\leq j\leq n$}\right\}\right|}{|S_{\lm_1}|\cdots|S_{\lm_n}|}>\D^*(A)^n-\VE\eee
with the understanding that $x_0=0$.
\end{itemize}
\end{thm}

In Theorem \ref{unpinnedUD} above, and throughout the paper, we use the notation $\A\ll \B$ to denote that $\A\leq c\B$ for some suitably small constant $c>0$.


Theorems \ref{pinned} and \ref{pinnedTree} likewise reduce to the following analogous result for uniformly distributed sets.

\begin{thm}[Theorems \ref{pinned} and \ref{pinnedTree} for Uniformly Distributed Sets]\label{pinnedUD}\

Let $\VE>0$, $0<\eta\ll\VE^3$, and $A\subseteq \Z^d$ with $d\geq 5$ be $\eta$-uniformly distributed.

\begin{itemize}
\item[\text{(i)}]
There exist $\Lm_0=\Lm_0(A,\eta)$ such that for any $\Lm_1\geq \Lm_0$ there exists a fixed $x\in A$ such that
\bee
\frac{|A\cap(x+S_{\lambda})|}{|S_\lm|}
>\D^*(A)-\VE\quad\text{for all integers} \quad \Lm_0\leq\lm\leq\Lm_1.\eee


\item[\text{(ii)}]
In fact, for any  down-labeled tree $\Gamma$ on $n+1$ vertices there exist $\Lm_0=\Lm_0(A,\eta)$ such that for any  $\Lambda_1\geq \Lm_0$ there exists a fixed $x\in A$ such that
\bee
\frac{\left|\left\{\{x_1,\dots,x_n\}\subseteq A-x\,:\, \text{$|x_j-x_{i_\Gamma(j)}|^2=\lm_j$ for $1\leq j\leq n$}\right\}\right|}{|S_{\lm_1}|\cdots|S_{\lm_n}|}>\D^*(A)^n-\VE\eee
for all integers $\lambda_1,\dots,\lm_n$ that satisfy  \ $\Lm_0\leq\lambda_1,\dots,\lm_n\leq\Lm_1$, with the understanding that $x_0=0$.
\end{itemize}
\end{thm}



\subsection{Compact variants of Theorems \ref{unpinnedUD} and \ref{pinnedUD}}

We shall now show that
Theorems \ref{unpinnedUD}  and \ref{pinnedUD} can in turn can be directly deduced from analogous \emph{compact} variants, namely Corollary \ref{Cor0} and Proposition \ref{Propn1} below. 

In what follows we shall  use $1_B$ to denote the characteristic function of any $B\subseteq\Z^d$ and define
\[\sigma_{\lm}=|S_{\lm}|^{-1}1_{S_{\lm}}.\]


First we introduce a second related notion of uniformity.

\begin{defn}[Definition of $(\eta,L)$-uniform distribution] 
Let $\eta>0$ and $q_\eta\ll \eta^2 L\ll \eta^4N.$ 

We define
$A\subseteq Q_N$ to be \emph{$(\eta,L)$-uniformly distributed} if \[\frac{1}{|Q_N|}\sum_{t\in Q_N}\left|\frac{|A\cap (t+Q_{q_\eta,L})|}{|Q_{q_\eta,L}|}-\frac{|A|}{|Q_N|}\right|^2\leq \eta^2,\]
where as before $Q_N$ denotes the discrete cube $[-N/2,N/2]^d\cap\Z^d$ and now
$Q_{q,L}:=Q_L\cap (q \Z)^d$.
\end{defn}

\smallskip

\begin{propn}[Average count of distances and trees in uniformly distributed subsets of $Q_N$]\label{Propn0}\

Let $\Gamma$  be a down-labeled tree on $n+1$ vertices.
If $\eta>0$ and $A\subseteq Q_N\subseteq \Z^d$ with $d\geq5$ 
 is
$(\eta,L)$-uniformly distributed, then 
for all integers $\lambda_1,\dots,\lm_n$ that satisfy $\eta^{-4}L^2\leq \lambda_1,\dots,\lm_n\leq \eta^{4} N^2$ one has
\bee
\frac{1}{|Q_N|}\sum_{x\in\Z^d}1_A(x)\!\!\!\!\sum_{x_1,\dots,x_n\in\Z^d}\prod_{j=1}^n1_A(x-x_j)\,\sigma_{\lm_j}(x_j-x_{i_\Gamma(j)})=\left(\frac{|A|}{|Q_N|}\right)^{n+1} + \ O(n\,\eta)
\eee
 with the understanding that $x_0=0$. 
 
 In particular, by taking $n=1$ above, one obtains that
\bee
\frac{1}{|Q_N|}\sum_{x\in\Z^d}1_A(x)\sum_{x_1\in\Z^d}1_A(x-x_1)\,\sigma_{\lm}(x_1)=\left(\frac{|A|}{|Q_N|}\right)^2 +O(\eta)
\eee
for all integers $\lambda$ that satisfy $\eta^{-4}L^2\leq \lambda\leq \eta^{4} N^2$.
\end{propn}

It is easy to see that Proposition \ref{Propn0} immediately implies the following

\begin{cor}[Unpinned distances and trees in uniformly distributed subsets of $Q_N$]\label{Cor0}\

Let $\VE>0$ and $A\subseteq Q_N\subseteq \Z^d$ with $d\geq5$ 
be
$(\eta,L)$-uniformly distributed for some $\eta>0$.

\begin{itemize}
\item[\text{(i)}]
If $0<\eta\ll\VE^2$, then for all integers $\lm$ that satisfy $\eta^{-4}L^2\leq \lm\leq \eta^{4} N^2$
there exists $x\in A$ such that
\[\frac{|A\cap(x+S_{\lambda})|}{|S_\lm|}=\sum_{x_1\in\Z^d}1_A(x-x_1)\,\sigma_{\lm}(x_1)>\frac{|A|}{|Q_N|}-\VE.\]


\item[\text{(ii)}]
If $0<\eta\ll\VE^2/n$, then for any  down-labeled tree $\Gamma$ on $n+1$ vertices and
 integers $\lambda_1,\dots,\lm_n$ that satisfy $\eta^{-4}L^2\leq \lambda_1,\dots,\lm_n\leq \eta^{4} N^2$ there exist $x\in A$ for which
\bee
\sum_{x_1,\dots,x_n\in\Z^d}\prod_{j=1}^n1_A(x-x_j)\,\sigma_{\lm_j}(x_j-x_{i_\Gamma(j)})>\left(\frac{|A|}{|Q_N|}\right)^{n}-\VE
\eee
with the understanding that $x_0=0$.
\end{itemize}
\end{cor}


We will ultimately also establish the following ``pinned" variant of Corollary \ref{Cor0}.

\begin{propn}[Pinned distances and trees in uniformly distributed subsets of $Q_N$]\label{Propn1}\

Let $\VE>0$, $0<\eta\ll\VE^3$, and $A\subseteq Q_N\subseteq \Z^d$ with $d\geq5$ 
be
$(\eta,L)$-uniformly distributed.

\begin{itemize}
\item[\text{(i)}]
There exists $x\in A$ such that
\[\frac{|A\cap(x+S_{\lambda})|}{|S_\lm|}>\frac{|A|}{|Q_N|}-\VE\]
for all integers $\lm$ that satisfy $\eta^{-4}L^2\leq \lm\leq \eta^{4} N^2$.

\smallskip

\item[\text{(ii)}]
In fact, for any  down-labeled tree $\Gamma$ on $n+1$ vertices there exist $x\in A$ such that
\bee
\sum_{x_1,\dots,x_n\in\Z^d}\prod_{j=1}^n1_A(x-x_j)\,\sigma_{\lm_j}(x_j-x_{i_\Gamma(j)})>\left(\frac{|A|}{|Q_N|}\right)^{n}-\VE
\eee
for all integers $\lambda_1,\dots,\lm_n$ satisfying $\eta^{-4}L^2\leq \lambda_1,\dots,\lm_n\leq \eta^{4} N^2$, with the understanding that $x_0=0$.
\end{itemize}
\end{propn}

\subsection{Reduction of Theorems \ref{unpinnedUD} and \ref{pinnedUD} to Corollary \ref{Cor0} and Proposition \ref{Propn1}}\

The task of deducing Theorems \ref{unpinnedUD} and \ref{pinnedUD} from Corollary \ref{Cor0} and Proposition \ref{Propn1} respectively simply 
amounts to establishing the following precise relationship between our two notions of uniform distribution.

\begin{lem}\label{udlem}
Let $\eta>0$. If $A\subseteq \Z^d$ with $\D^*(A)>0$ is $\eta$-uniformly distributed, then there exists a positive integer $L=L(A,\eta)$ and an arbitrarily large integer $N$ with  $N\geq\eta^{-4}L$  such that the set
$(A-t_0)\cap Q_N$ satisfies
\bee
\frac{|(A-t_0)\cap Q_N|}{|Q_N|}>\de^\ast (A)\eee
for some   $t_0\in \Z^d$ and simultaneously has the property that it
is $(C\eta,L)$-uniformly distributed for some $C>0$.

\end{lem}

\begin{proof}
Since $A\subseteq\Z^d$ is $\eta$-uniformly distributed we know there exists a positive integer  $L=L(A,\eta)$ such that 
\be\label{i}\frac{|A\cap (t+Q_{q_\eta,L})|}{|Q_{q_\eta,L}|}\leq (1+\eta^4/3)\,\de^\ast (A)\ee
for all $t\in\Z^d$.  Since $\D^*(A)>0$  we further know that there exist arbitrarily large $N\in\N$ such that
\be\label{ii}\frac{|A\cap(t_0+Q_N)|}{|Q_N|}\geq (1-\eta^4/3)\,\de^\ast (A)\ee
for some $t_0\in\Z^d$.
Combining (\ref{i}) and (\ref{ii}) we see  there exist $N\in\N$ with $N\geq\eta^{-4}L$ and $t_0\in\Z^d$ such that
\bee
\frac{|A\cap (t+Q_{q_\eta,L})|}{|Q_{q_\eta,L}|}\leq(1+\eta^4)\frac{|A\cap(t_0+Q_N)|}{|Q_N|}
\eee
for all $t\in\Z^d$. 
Setting $A':=(A-t_0)\cap Q_N$ we further note that since $A'\cap(t+Q_{q_\eta,L})$ is only supported in $Q_N+Q_L$ it follows that
\bee\label{Q-QL}
|A'|=\sum_{t\in\Z^d}\frac{|A'\cap(t+Q_{q_\eta,L})|}{|Q_{q_\eta,L}|}=\sum_{t\in Q_N}\frac{|A'\cap(t+Q_{q_\eta,L})|}{|Q_{q_\eta,L}|}+O(L/N),
\eee
from which one can easily deduce that 
\bee\label{most t's}
\frac{1}{|Q_N|}\Bigl|\Bigl\{t\in Q_N\,:\, \frac{|A'\cap(t+Q_{q_\eta,L})|}{|Q_{q_\eta,L}|}\leq (1-\eta^2)\,\frac{|A'|}{|Q_N|}\Bigr\}\Bigr|=O(\eta^2)
\eee
provided $L/N\ll\eta^2$ and hence that $A'$ is $(C\eta,L)$-uniformly distributed for some $C>0$. 
\end{proof}

We are thus left with  proving Propositions  \ref{Propn0} and \ref{Propn1}. These proofs are presented in  Sections \ref{P0Proof} and \ref{P1Proof} below. 


\section{Proof of Proposition \ref{Propn0}}\label{P0Proof}

\subsection{Reduction to a Generalized von-Neumann Inequality}\

Let $Q_N$ denote the discrete cube $[-N/2,N/2]^d\cap\Z^d$ with $d\geq5$.

\begin{defn}[Counting Function for Distances and Trees]\

\begin{itemize}
\item[\text{(i)}]
For $1\ll \lm\ll N^2$ and functions
$f_0,f_{1}:Q_N\to[-1,1]$ we define
\bee
T(f_0,f_1)(\lm)= \frac{1}{|Q_N|}\sum_{x\in\Z^d}f_0(x)\sum_{x_1\in\Z^d}f_1(x-x_1)\,\sigma_{\lm}(x_1).
\eee
\item[\text{(ii)}]
While, for any given down-labeled tree $\Gamma$   on $n+1$ vertices, $1\leq m\leq n$, $1\ll \lm_1,\dots,\lm_m\ll N^2$ and functions
$f_0,f_{1},\dots,f_m:Q_N\to[-1,1]$, we define
\bee
T_{\Gamma,m}(f_0,f_1,\dots,f_m)(\lm_1,\dots,\lm_m)= \frac{1}{|Q_N|}\sum_{x\in\Z^d}f_0(x)\!\!\!\!\sum_{x_1,\dots,x_m\in\Z^d}\prod_{j=1}^m f_j(x-x_j)\,\sigma_{\lm_j}(x_j-x_{i_\Gamma(j)})
\eee
with the understanding that $x_0=0$.
\end{itemize}
\end{defn}


\begin{defn}[$U^1(q,L)$-norm] 
For $1\ll q\ll L\ll N$ and functions $f:Q_N\to\R$ we define
\be\label{norm}
\|f\|_{U^1(q,L)}=\Bigl(\frac{1}{|Q_N|}\sum_{t\in \Z^d}|f*\chi_{q,L}(t)|^2\Bigr)^{1/2}
\ee
where $\chi_{q,L}$ denotes the  normalized characteristic function of the cubes $Q_{q,L}:=Q_L\cap (q \Z)^d$, namely
\be\label{normalizedchi}
\chi_{q,L}(x)=\begin{cases}
\left(\frac{q}{L}\right)^{d}& \ \ \text{if} \ \ x\in (q\Z)^d\,\cap [-\frac{L}{2},\frac{L}{2}]^d\\
 \ 0& \ \ \text{otherwise}
\end{cases}.
\ee
\end{defn}

In (\ref{norm}) above and in the sequel we denote the convolution $f*g$ of two functions $f$ and $g$ by
\[f*g(x):=\sum_{y\in\Z^d}f(x-y)g(y).\]

We note that the $U^1(q,L)$-norm measures the mean square oscillation of a function with respect to cubic grids of size $L$ and gap $q$.
It is a simple observation, that we record precisely below, that sets $A\subseteq Q_N$ that are $(\eta,L)$-uniformly distributed have the property that their ``balance functions" have small $U^1(q_\eta,L)$-norm.

\begin{lem}\label{ud} Let $\eta>0$ and  $1\ll L\ll\eta^2 N$. 

If $A\subseteq Q_N$ is $(\eta,L)$-uniformly distributed, then $\|f_A\|_{U^1(q_\eta,L)}\leq 2\eta$ where  $f_A=1_A-\frac{|A|}{|Q_N|}1_{Q_N}$. 
\end{lem}

In light of Lemma \ref{ud} we see that the  engine that drives our proof of Proposition \ref{Propn0} for $n=1$, and thus our short proof of Theorem \ref{unpinned}, via Part (i) of Corollary \ref{Cor0}, is the following ``generalized von-Neumann inequality''.

\begin{lem}[Generalized von-Neumann]\label{GvN0}
Let $\eta>0$, and $\lm$, $L$, and $N$ be integers with $\eta^{-4}L^2\leq \lm\leq \eta^{4} N^2$.

Given any functions
$f_0,f_1:Q_N\to[-1,1]$ 
on $Q_N\subseteq\Z^d$ with $d\geq5$ we have
\bee
\left|T(f_0,f_1)(\lm)\right|\leq\|f_1\|_{U^1(q_\eta,L)}+O(\eta).
\eee
\end{lem}

In fact, as we shall see below, this result is also sufficient to establish Proposition \ref{Propn0} for any $n$ and hence Theorem \ref{unpinnedTree} via Part (ii) of Corollary \ref{Cor0}.

\comment{
\begin{lem}[Generalized von-Neumann]\label{GvN0}

Let $\Gamma$  be a down-labeled tree on $n+1$ vertices and $1\leq m\leq n$.

For any $\eta>0$, integers $ \lm_1,\dots,\lm_m$ that satisfy $\eta^{-4}L^2\leq \lm_1,\dots,\lm_m\leq \eta^{4} N^2$, and functions
\[f_0,f_{1},\dots,f_m:Q_N\to[-1,1]\]
we have that
\bee
\left|T_{\Gamma,m}(f_0,f_1,\dots,f_m)(\lm_1,\dots,\lm_m)\right|\leq\|f_m\|_{U^1(q_\eta,L)}+O(\eta).
\eee
\end{lem}
}

\begin{proof}[Proof of Proposition \ref{Propn0}]
Let $\eta>0$ and $A\subseteq Q_N\subseteq \Z^d$ with $d\geq5$ 
be
$(\eta,L)$-uniformly distributed.

We let $\A=|A|/|Q_N|$ and note that Lemma \ref{ud} ensures that $ \|f_A\|_{U^1(q_\eta,L)}\leq 2\eta$ where $f_A=1_A-\A1_{Q_N}$. \

\noindent
\emph{Case $n=1$}:
 In this case Proposition \ref{Propn0}  follows immediately from Lemma \ref{GvN0} since
\[T(1_A,1_A)(\lm)=\A \, T(1_A,1_{Q_N})+T(1_A,f_A)(\lm)=\A^2+\|f_A\|_{U^1(q_\eta,L)}+O(\eta)\]
for all integers $\lm$ that satisfy $\eta^{-4}L^2\leq \lm\leq \eta^{4} N^2$. 

\noindent
\emph{Case $n\geq2$}:
We first fix $2\leq m\leq n$ and make two key observations.

Since $0\leq i_\Gamma(j)<j$ and  $f_j:Q_N\to[-1,1]$ for all $1\leq j\leq m$, and no function $\sigma_{\lm_j}$ involves the variable $x$, it follows that
\be\label{DistToTree}
|T_{\Gamma,m}(f_0,f_1,\dots,f_m)(\lm_1,\dots,\lm_m)|\leq \sum_{x_1,\dots,x_{m-1}\in\Z^d}\prod_{j=1}^{m-1}\sigma_{\lm_j}(x_j-x_{i_\Gamma(j)}) \ \left|T(\tau_{x_{i_\Gamma(m)}}f_0,f_m)(\lm_m)\right|
\ee
where $\tau_{x_{i_\Gamma(m)}}f_0(x)=f_0(x+x_{i_\Gamma(m)})$.
We also note that 
\be\label{error}
T_{\Gamma,m}(f_0,f_1,\dots,f_{m-1},1_{Q_N})(\lm_1,\dots,\lm_m)=T_{\Gamma,m-1}(f_0,f_1,\dots,f_{m-1})(\lm_1,\dots,\lm_{m-1})+O(\sqrt{\lm_m}/N).
\ee

It then follows from Lemma \ref{GvN0}, together with observations (\ref{DistToTree}) and (\ref{error}) above, that
\begin{align}T_{\Gamma,n}(1_A,\dots,1_A)(\lm_1,\dots,\lm_n)&=\A^{n}\,T_{\Gamma,1}(1_A,1_{Q_N})(\lm_1)\nonumber \\
&\quad\quad\quad\quad+\sum_{m=1}^n \A^{n-m}\,T_{\Gamma,m}(1_A,\dots,1_A,f_A)(\lm_1,\dots,\lm_m)  +O(n\eta^2)\label{induct} 
\\
&=\A^{n+1}  + O(n\, \|f_A\|_{U^1(q_\eta,L)}) +O(n\,\eta)\nonumber
\end{align}
 for all integers $\lambda_1,\dots,\lm_n$ that satisfy $\eta^{-4}L^2\leq \lambda_1,\dots,\lm_n\leq \eta^{4} N^2$.
\end{proof}

In order to prove Proposition \ref{Propn0} we thus left with the final task of establishing Lemma \ref{GvN0}. 

\subsection{Proof of  Lemma \ref{GvN0}}


For any $f:Q_N\to[-1,1]$ we define its \emph{Fourier transform} $\widehat{f}:\T^d\rightarrow\C$ by
\[\widehat{f}(\xi)= \sum\limits_{x\in\Z^d}f(x)e^{-2\pi i x\cdot\xi}\]
noting that the support assumption on $f$ ensures that the series defining $\widehat{f}$ converges uniformly to a continuous function on the torus $\T^d$, which we will freely identify with the unit cube $[0,1)^d$ in $\R^d$. 

\comment{
Fix $0\leq m\leq n$.  
Since $0\leq i_\Gamma(j)<j$ and  $f_j:Q_N\to[-1,1]$ for all $1\leq j\leq m$, and no function $\sigma_{\lm_j}$ involves the variable $x$, it follows that
\begin{align*}
|T_{\Gamma,m}(f_0,f_1,\dots,f_m)&(\lm_1,\dots,\lm_m)|\leq \\&\sum_{x_1,\dots,x_{m-1}\in\Z^d}\prod_{j=1}^{m-1}\sigma_{\lm_j}(x_j-x_{i_\Gamma(j)}) \ \frac{1}{|Q_N|}\sum_{x\in\Z^d}\left|\sum_{x_m\in\Z^d} f_m(x-x_m)\,\sigma_{\lm_m}(x_m-x_{i_\Gamma(m)})\right|
\end{align*}

It therefore follows, via an application of Cauchy-Schwarz and basic properties of the Fourier transform, that 
\[|T_{\Gamma,m}(f_0,f_1,\dots,f_m)(\lm_1,\dots,\lm_m)|^2\leq
\frac{1}{|Q_N|}\int|\widehat{f_m}(\xi)|^2|\widehat{\sigma_{\lambda_m}}(\xi)|^2\,d\xi\]
where
\be\label{es}\widehat{\sigma_{\lambda}}(\xi):=\frac{1}{|S_{\lambda}|}\sum_{x\in S_{\lambda}}e^{-2\pi i x\cdot\xi}.\ee
}

It is easy to verify, using Cauchy-Schwarz and basic properties of the Fourier transform,  that 
\[\left|T(f_0,f_1)(\lm)\right|^2\leq
\frac{1}{|Q_N|}\int|\widehat{f_1}(\xi)|^2|\widehat{\sigma_{\lambda}}(\xi)|^2\,d\xi\]
where
\be\label{es}\widehat{\sigma_{\lambda}}(\xi):=\frac{1}{|S_{\lambda}|}\sum_{x\in S_{\lambda}}e^{-2\pi i x\cdot\xi}.\ee

It is clear that whenever $|\xi|^2\ll \lm^{-1}$ there can be no cancellation in the exponential sum (\ref{es}), in fact it is easy to verify that the same is also true whenever $\xi$ is \emph{close} to a rational point with \emph{small} denominator.
The following proposition is a precise formulation of the fact that this is the only obstruction to cancellation. 

\begin{propn}[Key exponential sum estimates, Proposition 1 in \cite{M1}]\label{KeyU}
Let $\eta>0$. 
If $\lambda\geq C\eta^{-4}$ and \[\xi\notin 
\bigl(q_\eta^{-1}\mathbb{Z}\bigr)^d+\{\xi\in\mathbb{R}^d\,:\,|\xi|^2\leq\eta^{-1}\lambda^{-1}\},\] 
then \[\Bigl|\frac{1}{|S_{\lambda}|}\sum_{x\in S_{\lambda}}e^{-2\pi i x\cdot\xi}\Bigr|\leq \eta.\] 
\end{propn}

We now define $\psi_{q_\eta,L}$ indirectly via the identity
\[\widehat{\psi_{q_\eta,L}}(\xi):=\widehat{\chi_{q_\eta,L}}(\xi)^2.\]
Since the definition of $\chi_{q_\eta,L}$ in (\ref{normalizedchi}) above clearly implies that  
\[\text{(i) \ $0\leq\widehat{\psi_{q_\eta,L}}(\xi)\leq1$ for all $\xi\in\T^d$ \quad and \quad
(ii) \ $\widehat{\psi_{q_\eta,L}}(\ell/q_\eta)=1$ for all $\ell\in\Z^d$}\]
it follows that 
\[
0\leq1-\widehat{\psi_{q_\eta,L}}(\xi)\ll L|\xi-\ell/q_\eta|
\]
for all $\xi\in\T^d$ and  $\ell\in\Z^d$. In particular we note that 
\be\label{1-psi}
|1-\widehat{\psi_{q_\eta,L}}(\xi)|\ll\eta \quad\text{if \ $|\xi-\ell/q_\eta|\leq\eta^{-1/2}\lm^{-1/2}$ \  for some $\ell\in\Z^d$.}
\ee
while
 Proposition \ref{KeyU} ensures that
\be\label{sigma}
|\widehat{\sigma_{\lambda}}(\xi)|\leq\eta \quad\text{if \ $|\xi-\ell/q_\eta|>\eta^{-1/2}\lm^{-1/2}$ \  for all $\ell\in\Z^d$}.
\ee

Hence, if we write
\[ |\widehat{\sigma_{\lambda}}(\xi)|^2=|\widehat{\sigma_{\lambda}}(\xi)|^2\widehat{\psi}_{q_\eta,L}(\xi)  +|\widehat{\sigma_{\lambda}}(\xi)|^2(1-\widehat{\psi}_{q_\eta,L}(\xi))\]
\\
use the fact that $|\widehat{\sigma_{\lambda}}(\xi)|\leq1$ for all $\xi\in\T^d$ and appeal to Plancherel we can deduce that
\[\left|T(f_0,f_1)(\lm)\right|^2\leq\frac{1}{|Q_N|}\int|\widehat{f_1}(\xi)|^2\,\widehat{\chi}_{q_\eta,L}(\xi)^2\,d\xi+O(\eta^2)=\|f_1\|_{U^1(q_\eta,L)}^2+O(\eta^2)\]
which completes the proof of Lemma \ref{KeyU}.\qed

\section{Proof of Proposition \ref{Propn1}}\label{P1Proof}

Let $Q_N$ denote the discrete cube $[-N/2,N/2]^d\cap\Z^d$.

\begin{defn}[Discrete Spherical Averages]

Let 
$f:Q_N\to\R$ be any function. 

For any integer $\lm$ with $1\ll \lm\ll N^2$ we define the \emph{discrete spherical average}
\bee
\mathcal{A}_{\lm}(f)(x):=f*\sigma_\lm(x)=\frac{1}{|S_\lm|}\sum_{y\in S_\lm}f(x-y).
\eee
\end{defn}

In Sections \ref{4.1} and \ref{4.2} below we reduce Proposition \ref{Propn1} to
the \emph{Discrete Spherical Maximal Function Theorem} of Magyar, Stein  and Wainger \cite{MSW}, see Proposition \ref{MSW}, and a new ``mollified variant" thereof, namely Proposition \ref{mollified}. The  statement and proof of Proposition \ref{mollified} is  presented in Section \ref{MollyProof}.

\subsection{Proof of Part (i) of Proposition \ref{Propn1}}\label{4.1}\

Let $\VE>0$ and $0<\eta\ll\VE^3$. Suppose, contrary to Part (i) of Proposition \ref{Propn1}, that there  exists a set  $A\subseteq Q_N\subseteq \Z^d$ with $d\geq5$ 
and $\A=|A|/|Q_N|>0$ that it is
$(\eta,L)$-uniformly distributed, but has the property that
for every $x\in A$ there exists an integer $\lm$ with
$\eta^{-4}L^2\leq \lm\leq \eta^{4} N^2$ such that
\[\mathcal{A}_{\lambda}(1_A)(x)=\frac{|A\cap(x+S_{\lambda})|}{|S_\lm|}\leq\A-\VE.\]

It easily follows that for every $x\in A$ there exists an integer $\lm$ with
$\eta^{-4}L^2\leq \lm\leq \eta^{4} N^2$ such that
\bee
\mathcal{A}_{\lambda}(f_A)(x)=-\VE+O(\sqrt{\lm}/N)
\eee
where $f_A=1_A-\A1_{Q_N}$.
Hence for every $x\in A$ we may conclude that
 \be\label{max}
\mathcal{A}_*(f_A)(x)\geq \VE/2
\ee
where  for any function $f:\Z^d\to\R$, $\mathcal{A}_*(f)$ denotes the \emph{discrete spherical maximal function} defined by
\bee
\mathcal{A}_*(f)(x):=\sup_{\eta^{-4}L^2\leq \lm\leq \eta^{4} N^2}\left|\mathcal{A}_{\lambda}(f)(x)\right|.\eee

\begin{propn}[$\ell^2$-Boundedness of the Discrete Spherical Maximal Function \cite{MSW}]\label{MSW}
If $d\geq5$, then
\[\sum_{x\in\Z^d}|\mathcal{A}_*(f)(x)|^2\leq C \sum_{x\in\Z^d}|f(x)|^2.\]
\end{propn}

Since (\ref{max}) implies, after an application of Cauchy-Schwarz, the inequality
\be\label{littlemain}
\frac{\A\,\VE}{2}\leq\frac{1}{|Q_N|}\sum_{x\in\Z^d}1_A(x)\mathcal{A}_*(f_A)(x)\leq \A^{1/2}\,\Bigl(\frac{1}{|Q_N|}\sum_{x\in\Z^d}|\mathcal{A}_*(f_A)(x)|^2\Bigr)^{1/2}
\ee
it follows that
\be\label{maininequal}
\frac{\A^{1/2}\,\VE}{2}\leq \Bigl(\frac{1}{|Q_N|}\sum_{x\in\Z^d}|\mathcal{A}_*(f_A*\chi_{q_\eta,L})(x)|^2\Bigr)^{1/2} +  \Bigl(\frac{1}{|Q_N|}\sum_{x\in\Z^d}|\mathcal{A}_*(f_A-f_A*\chi_{q_\eta,L})(x)|^2\Bigr)^{1/2}.
\ee
In light of Proposition \ref{MSW} it follows that the first sum above satisfies
\bee
\Bigl(\frac{1}{|Q_N|}\sum_{x\in\Z^d}|\mathcal{A}_*(f_A*\chi_{q_\eta,L})(x)|^2\Bigr)^{1/2}\leq C \Bigl(\frac{1}{|Q_N|}\sum_{t\in \Z^d}|f*\chi_{q,L}(t)|^2\Bigr)^{1/2}= C\,\|f_A\|_{U^1(q_\eta,L)}\leq 2C \eta.
\eee

Estimate (\ref{maininequal}) will therefore lead to a contradiction, if $\eta$ is chosen sufficiently small with respect to $\VE^3$, and hence complete the proof of Proposition \ref{Propn1} if we establish that the second sum in (\ref{maininequal}) satisfies
\be\label{sum2}
\Bigl(\frac{1}{|Q_N|}\sum_{x\in\Z^d}|\mathcal{A}_*(f_A-f_A*\chi_{q_\eta,L})(x)|^2\Bigr)^{1/2}\leq C\eta^{1/3}\A^{1/2}
\ee
for some absolute constant $C>0$. 
Estimate (\ref{sum2}) follows immediately from  Proposition \ref{mollified} in Section \ref{MollyProof} below.

\subsection{Proof of Part (ii) of Proposition \ref{Propn1}}\label{4.2}\

Let $\VE>0$ and $0<\eta\ll\VE^3$. Suppose, contrary to Part (ii) of Proposition \ref{Propn1}, that there exists a set  $A\subseteq Q_N\subseteq \Z^d$ with $d\geq5$ 
and $\A=|A|/|Q_N|>0$ that it is
$(\eta,L)$-uniformly distributed, but has the property that
for every $x\in A$ there exist  integers  $\lambda_1,\dots,\lm_n$ satisfying $\eta^{-4}L^2\leq \lambda_1,\dots,\lm_n\leq \eta^{4} N^2$ such that
\be\label{contrary}
\sum_{x_1,\dots,x_n\in\Z^d}\prod_{j=1}^n1_A(x-x_j)\,\sigma_{\lm_j}(x_j-x_{i_\Gamma(j)})\leq\A^{n}-\VE
\ee
where $x_0=0$. As in (\ref{induct}) we may write the left hand side of (\ref{contrary}) as
\[\A^n+\sum_{m=1}^n \A^{n-m}\,\left[\sum_{x_1,\dots,x_{m-1}\in\Z^d}\prod_{j=1}^{m-1}1_A(x-x_j)\,\sigma_{\lm_j}(x_j-x_{i_\Gamma(j)})\sum_{x_{m}\in\Z^d} f_A(x-x_m)\,\sigma_{\lm_m}(x_m-x_{i_\Gamma(m)})+O(\eta^2)\right]\]
 where $f_A=1_A-\A1_{Q_N}$. 
 It therefore follows, from  the pigeonhole principle, 
 that  for every $x\in A$ there must exist $2\leq m\leq n$ and integers  $\lambda_1,\dots,\lm_{m-1}$ satisfying $\eta^{-4}L^2\leq \lambda_1,\dots,\lm_{m-1}\leq \eta^{4} N^2$ such that
\bee
\sum_{x_1,\dots,x_{m-1}\in\Z^d}\prod_{j=1}^{m-1}\sigma_{\lm_j}(x_j-x_{i_\Gamma(j)})\left|\mathcal{A}_{*}(f_A)(x-x_{i_\Gamma(m)})\right|\geq \frac{\VE}{2}
\eee
 provided  $\eta^2\ll\VE$.
 Since 
\bee
\sum_{x\in\Z^d}\sum_{x_1,\dots,x_{m-1}\in\Z^d}\prod_{j=1}^{m-1}\sigma_{\lm_j}(x_j-x_{i_\Gamma(j)})\left|\mathcal{A}_{*}(f_A)(x-x_{i_\Gamma(m)})\right|^2 = \sum_{x\in\Z^d}\left|\mathcal{A}_{*}(f_A)(x)\right|^2
\eee
for any choice of integers $\lm_1,\dots,\lm_{m-1}$, it follows by Cauchy-Schwarz  that
\[
\frac{1}{|Q_N|}\sum_{x\in\Z^d}1_A(x)\!\!\!\!\sum_{x_1,\dots,x_{m-1}\in\Z^d}\prod_{j=1}^{m-1}\sigma_{\lm_j}(x_j-x_{i_\Gamma(j)})\left|\mathcal{A}_{*}(f_A)(x-x_{i_\Gamma(m)})\right|\leq \A^{1/2}\,\Bigl(\frac{1}{|Q_N|}\sum_{x\in\Z^d}|\mathcal{A}_*(f_A)(x)|^2\Bigr)^{1/2}.
\]

We can therefore conclude that 
\be\label{bound}
\frac{1}{|Q_N|}\sum_{x\in\Z^d}|\mathcal{A}_*(f_A)(x)|^2\geq \frac{\A\,\VE^2}{4}
\ee
and estimate which, in will in light of  Proposition \ref{mollified} (as  described after estimate (\ref{littlemain}) in Section \ref{4.1} above), leads to a contradiction if $\eta$ is chosen sufficiently small with respect to $\VE^3$.

\section{A ``Mollified" Discrete Spherical Maximal Function Theorem}\label{MollyProof}

Let $\eta>0$ and $\lm$, $L$, and $N$ be integers that satisfy $\eta^{-4}L^2\leq \lm\leq \eta^{4} N^2$. 
For functions
$f:Q_N\to[-1,1]$ we now define
\[\mathcal{A}_{\lm,\eta}(f)(x):=\mathcal{A}_{\lambda}(f-f*\chi_{q_\eta,L})(x)=f*(\sigma_{\lm}-\sigma_\lm \ast \chi_{q_\eta,L})(x)
\]
where $\sigma_\lm=\dfrac{1}{|S_\lm|}1_{S_\lm}$, and introduce the corresponding \emph{``mollified" discrete spherical maximal function}
\be
\mathcal{A}_{*,\eta}(f)(x):=\sup_{\eta^{-4}L^2\leq \lm\leq \eta^{4} N^2}\left|\mathcal{A}_{\lambda,\eta}(f)(x)\right|.\ee

We note that the convolution operator $\mathcal{A}_{\lm,\eta}$ corresponds to the Fourier multiplier $\widehat{\si_{\lm,\eta}}:=
\widehat{\si_\lm} (1-\widehat{\chi_{q_\eta, L}}).$

\begin{propn}[$\ell^2$-Decay of the ``Mollified" Discrete Spherical Maximal Function]\label{mollified}
If $d\geq5$, then for any $\eta>0$ we have
\be\label{maxerror}
\sum_{x\in\Z^d}|\mathcal{A}_{*,\eta}(f)(x)|^2\leq C\eta^{2/3} \sum_{x\in\Z^d}|f(x)|^2.
\ee
\end{propn}

\begin{proof}[Proof of Proposition \ref{mollified}]


We follow the proof of Proposition \ref{MSW} as given in \cite{MSW}. For each $x\in\Z^d$ we now define
\bee
\mathcal{B}_\lm (f)(x)=\mathcal{A}_{\lm^2} (f)(x)
\eee
noting that when considering $\mathcal{B}_\lm$ we are now allowing all values of $\lm$ for which $\lm^2$ is an integer, and that \[\mathcal{B}_* (f)(x):=\sup_{\eta^{-2}L\leq \lm\leq \eta^{2} N}\mathcal{B}_\lm (f)(x)=\mathcal{A}_{*} (f)(x)\quad \text{and} \quad \mathcal{B}_{*,\eta}( f)(x)=\mathcal{B}_*(f-f*\chi_{q_\eta,L})(x).\]

We now recall the approximation to $\mathcal{B}_\lm$ given in Section 3 of \cite{MSW} as a convolution operator $\MM_\lm$ acting on functions on $\Z^d$ of the form
\vspace{-3pt}
\bee
\MM_\lm =c_d \,\sum_{q=1}^\infty \sum_{\substack{1\leq a\leq q\\(a,q)=1}} e^{-2\pi i\lm a/q} \MM_\lm^{a/q}
\eee
where for each reduced fraction $a/q$ the corresponding convolution operator $\MM_\lm^{a/q}$ has Fourier multiplier 
\bee
m_\lm^{a/q}(\xi):= \sum_{\ell\in\Z^k} G(a/q,\ell)\vp_q(\xi-\ell/q)\widetilde{\si}_\lm(\xi-\ell/q)\eee
with $\vp_q(\xi)=\vp(q\xi)$ a standard smooth cut-off function, $G(a/q,l)$ a normalized Gauss sum, and $\widetilde{\si}_\lm(\xi)=\widetilde{\si}(\lm\xi)$ where $\widetilde{\si}(\xi)$ is the Fourier transform (on $\mathbb{R}^d$) of the measure on the unit sphere in $\R^d$ induced by Lebesgue measure and normalized to have total mass $1$.
By Proposition 4.1 in \cite{MSW} we have 
\bee
\Bigl\|\sup_{\Lambda\leq \lm \leq 2\Lambda} |\mathcal{B}_\lm (f)-\MM_\lm (f)|\Bigr\|_{\ell^2(\Z^d)}\leq C \Lambda^{-1/2} \|f\|_{\ell^2(\Z^d)}\eee
provided $d\geq 5$.
Writing 
\[\MM_* (f):= \sup_{\eta^{-2}L\leq \lm\leq \eta^{2} N} |\MM_\lm (f)|\quad \text{and} \quad \MM_{*,\eta} (f):= \MM_* (f-f*\chi_{q_\eta,L})\]
 this implies
\bee
\|\mathcal{B}_{*,\eta} (f)-\MM_{*,\eta} (f)\|_{\ell^2} \leq C\,\eta L^{-1/2}\, \|f-f*\chi_{q_\eta,L}\|_{\ell^2}\leq C\,\eta L^{-1/2}\,\|f\|_{\ell^2}\eee
thus matters reduce to showing \eqref{maxerror} for the operator $\MM_{*,\eta}$.

For a given reduced fraction $a/q$ we now define the maximal operator
\bee
\MM_{*}^{a/q}(f):= \sup_{\eta^{-2}L\leq \lm\leq \eta^{2} N} |\MM_{\lm}^{a/q}(f)|\eee
where $\MM_{\lm}^{a/q}$ is the convolution operator with multiplier $m_\lm^{a/q}(\xi)$.
It is proved in Lemma 3.1 of \cite{MSW} that
\be\label{maxbound1}
\|\MM_{*}^{a/q}(f)\|_{\ell^2}\leq C q^{-d/2}\|f\|_{\ell^2}.
\ee

We will show here that if $q\leq C\eta^{-2/3}$, then
\be\label{maxbound2}
\|\MM_{*}^{a/q}(f-f*\chi_{q_\eta,L})\|_{\ell^2}\leq C \eta^{1/3} q^{-d/2}\|f\|_{\ell^2}.
\ee

Taking estimates \eqref{maxbound1} and \eqref{maxbound2} for granted, one obtains
\bee
\|\MM_{*}(f-f*\chi_{q_\eta,L})\|_{\ell^2}\, \ll\,\Bigl(\eta^{1/3}\sum_{1\leq q\leq C\eta^{-2/3}} q^{-d/2+1} +
\sum_{q\geq C\eta^{-2/3}} q^{-d/2+1}\Bigr)\, \|f\|_{\ell^2} \,\ll\,\eta^{1/3} \|f\|_{\ell^2}\eee
as required.
It thus remains to prove \eqref{maxbound2}. 

Writing $\vp_q(\xi)=\vp'_q(\xi)\vp_q(\xi)$, with a suitable smooth cut-off function $\vp'$, we can introduce the decomposition 
\bee
m_\lm^{a/q}(\xi)=\Bigl(\sum_{\ell\in\Z^k} G(a/q,\ell)\vp'_q (\xi-\ell/q)\Bigr) \,\Bigl(\sum_{\ell\in\Z^k} \vp_q (\xi-\ell/q)\widetilde{\si}(\xi-\ell/q)\Bigr)=:g^{a/q}(\xi)\, n^{q}_\lm(\xi),\eee
since for each $\xi$ at most one term in each of the above sums is non-vanishing.
Accordingly
\bee
\MM_{*}^{a/q}(f-f*\chi_{q_\eta,L})=G^{a/q}\ \NN_*^{q}(f-f*\chi_{q_\eta,L})\eee
where the maximal operator $\NN_*^{q}$ and the convolution operator $G_{a/q}$ correspond to the multipliers $n^{q}_\lm$ and $g^{a/q}$ respectively. 
Now by the standard Gauss sum estimate we have $|g^{a/q}(\xi)|\ll q^{-d/2}$ uniformly in $\xi$, hence 
\bee
\|G^{a/q}\ \NN_*^{q}(f-f*\chi_{q_\eta,L})\|_{\ell^2}\ll q^{-d/2} \|\NN_*^{q}(f-f*\chi_{q_\eta,L})\|_{\ell^2}.\eee

Thus by our choice $q_\eta:=\lcm\{1\leq q\leq C\eta^{-2}\}$ it remains to show that if $q$ divides $q_\eta$ then
\be\label{maxbound3}
\|\NN_{*}^{q}(f-f*\chi_{q_\eta,L})\|_{\ell^2}\ll \eta^{1/3} \|f\|_{\ell^2}.
\ee

As before we write $\NN_{*, \eta}^{q}(f)=\NN_{*}^{q}(f-f*\chi_{q_\eta,L})$, and note that this is a  maximal operator with multiplier
\bee
n_\lm^{q}(\xi) (1-\widehat{\chi_{q_{\eta,L}}})(\xi)=\sum_{\ell\in \Z^d} \vp_q(\xi-\ell/q) (1-\widehat{\chi_{q_{\eta,L}}})(\xi-\ell/q)\widetilde{\si}_\lm (\xi-\ell/q).\eee
For a fixed $q$, the multiplier $\vp_q(1-\widehat{\chi_{q_{\eta,L}}})\widetilde{\si}_\lm$ is supported on the cube $[-\frac{1}{2q}, \frac{1}{2q}]^d$ thus by Corollary 2.1 in \cite{MSW}
\[\|\NN_{*, \eta}^{q}\|_{{\ell^2}\to{\ell^2}} \leq C\,\| \widetilde{\NN}_{*, \eta}^{q}\|_{{L^2}\to{L^2}}\]
where $\widetilde{\NN}_{*, \eta}^{q}$ is the maximal operator corresponding to the multipliers $\vp_q(1-\widehat{\chi_{q_{\eta,L}}})\widetilde{\si}_\lm$, for $\eta^{-2}L\leq \lm\leq \eta^{2} N$, acting on $L^2(\R^d)$. 
By the definition of the function $\chi_{q_{\eta,L}}$
\[|1-\widehat{\chi_{q_{\eta,L}}}(\xi)|\ll \min\{1,L|\xi|\},\]
thus from Theorem 6.1 (with $j=1$) in \cite{HLM} we obtain
\[\|\widetilde{\NN}_{*, \eta}^{q}\|_{{L^2}\to{L^2}} \ll \left(\frac{L}{\eta^{-2}L}\right)^{1/6}=\eta^{1/3}\]
which establishes (\ref{maxbound3}) and completes the proof.
\end{proof}




\begin{thebibliography}{10}


\bibitem{B}
{\sc J. Bourgain}, {\em
A Szemer\'edi type theorem for sets of positive density in $\R^k$},
Israel J. Math. 54 (1986), no. 3, 307--316.

\bibitem{BU}{\sc K. Bulinski},
{\em Spherical Recurrence and locally isometric embeddings of trees into positive density subsets of $\mathbb{Z}^ d$},  Math. Proc. Cambridge Philos. Soc. 165 (2018), no. 2, 267-278


\bibitem{FKW}
{\sc H. Furstenberg, Y. Katznelson and B. Weiss}, {\em
Ergodic theory and configurations in sets of positive density},
Israel J. Math. 54 (1986), no. 3, 307--316.

\bibitem{HLM}
{\sc L. Huckaba, N. Lyall and \'A. Magyar}, {\em Simplices and sets of positive upper density in $\R^d$}, Proc. Amer. Math. Soc. 145 (2017), no. 6, 2335-2347

\bibitem{LM2}
{\sc N. Lyall and \'A. Magyar}, {\em Optimal polynomial recurrence},  Canad. J. Math. 65 (2013), no. 1, 171-194

\bibitem{M1}
{\sc \'A. Magyar}, {\em
On distance sets of large sets of integer points},
Israel J. Math. 164 (2008), 251--263.

\bibitem{MSW}
{\sc \'A. Magyar, E.M. Stein, S. Wainger},
{\em Discrete analogues in harmonic analysis: spherical averages},
Annals of Math., 155 (2002), 189-208


\end{thebibliography}
\end{document}